\documentclass{amsart}%
\usepackage{cite}
\usepackage[hypertex]{hyperref}
\usepackage{amssymb}
\usepackage{amsmath}
\usepackage{amsthm}
\usepackage{amsfonts}
\usepackage{graphicx}
\usepackage{color}
\usepackage{MnSymbol}
\usepackage{mathrsfs}
\usepackage[toc,page]{appendix}%

\usepackage{marginnote}

\newcommand{\Null}{\operatorname{Null}}

\newtheorem{thm}{Theorem}[section]
\newtheorem{cor}[thm]{Corollary}

\newtheorem{assumption}{Assumption}
\newtheorem{lem}[thm]{Lemma}
\newtheorem{nota}[thm]{Notation}
\newtheorem{prop}[thm]{Proposition}

\newtheorem{as}[thm]{Assumption}

\theoremstyle{definition}
\newtheorem{df}[thm]{Definition}

\newtheorem{example}{Example}[section]
\newtheorem{rem}[thm]{Remark}

\numberwithin{equation}{section}
\newtheorem*{acknowledgement}{Acknowledgement}

\begin{document}

\title[Brownian Motion on a Sub-Riemannian Manifold]{Weak Convergence to Brownian Motion on Sub-Riemannian Manifolds}
\author[Gordina]{Maria Gordina{$^{\dagger}$}}
\thanks{\footnotemark {$\dagger$} This research was supported in part by NSF Grant DMS-1007496.}
\address{Department of Mathematics\\
University of Connecticut\\
Storrs, CT 06269, USA} \email{maria.gordina@uconn.edu}
\author[Laetsch]{Thomas Laetsch}
\address{Department of Mathematics\\
University of Connecticut\\
Storrs, CT 06269, USA} \email{thomas.laetsch@uconn.edu}




\begin{abstract} This paper considers a classical question of approximation of Brownian motion by a random walk in the setting of a sub-Riemannian manifold $M$. To construct such a random walk we first address several issues related to the degeneracy of such a manifold. In particular, we define a family of sub-Laplacian operators naturally connected to the geometry of the underlining manifold. In the case when $M$ is a Riemannian (non-degenerate) manifold, we recover the Laplace-Beltrami operator. We then  construct the corresponding random walk, and  under standard  assumptions on the sub-Laplacian and $M$ we show that this random walk weakly converges to a process, horizontal Brownian motion,  whose infinitesimal generator is the sub-Laplacian. An example of the Heisenberg group equipped with a standard sub-Riemannian metric is considered in detail, in which case the sub-Laplacian we introduced is shown to be the sum of squares (H\"{o}rmander's) operator.

\end{abstract}

\maketitle
\tableofcontents

\section{Introduction}\label{s.Intro}
This paper describes a geometrically natural piecewise Hamiltonian-flow random walk in a sub-Riemannian manifold, which converges weakly to a horizontal Brownian motion on the manifold. In this setting we define a sub-Laplacian by averaging over the second-order directional derivatives in the directions of the Hamiltonian flow. In particular, in the case $\mathcal{H} = TM$, we recover the Laplace-Beltrami operator; in the case of the Heisenberg group equipped with a standard sub-Riemannian metric, we recover the sum of squares (H\"{o}rmander's) operator. While the current paper presents the probabilistic aspects of this construction, the geometric exploration of this sub-Laplacian can be found in \cite{GordinaLaetsch2014b}. As we will see in Section \ref{s4}, the sub-Laplacian we study is the one that generates the horizontal Brownian motion.

Over the last half century, Brownian motion on Riemannian manifolds has developed into a well-understood and rich theory. Much of this development relies heavily on the Riemannian structure as one can see from the monographs \cite{ElworthyBook, HsuBook2002}. There are two major ingredients which are canonical in the Riemannian case: the Riemannian volume $\mu$ and the corresponding Laplace-Beltrami operator $\Delta_{LB}$. Recall that  the Laplace-Beltrami operator is usually defined as $\operatorname{div} \operatorname{grad}$, where  $\operatorname{div}$ is defined with respect to the Riemannian volume $\mu$. From here, a Brownian motion on a Riemannian manifold can be described as a stochastic process with the infinitesimal generator  $\Delta_{LB}$.

This approach is not easily available in the sub-Riemannian case. There are several measures which might be used in lieu of the Riemannian volume such as the Hausdorff measure, Popp's measure (see \cite{MontgomeryBook2002,AgrachevBoscainGauthierRossi2009}), left or right Haar measure in the case of Lie groups. Each choice of the measure will lead to a possibly different sub-Laplacian, and therefore to a different Brownian motion.  A more detailed analysis of sub-Laplacians and natural choices of measures is presented in \cite{GordinaLaetsch2014b}.

Instead of making this choice,  we develop a more classical approach of constructing a Brownian motion as the limit of an appropriately-scaled random walk. Any complete list of references working in this direction on Riemannian manifolds would undoubtedly include the now-classic works  \cite{JorgensenE1975,Malliavin1975a}, and most relevant to our work, the {\em isotropic transport process} by M.~Pinsky \cite{Pinsky1976a}. Motivated by Pinsky's approach, the sub-Laplacian we construct is canonical {\em with respect to the limiting process of the random walk}. This sub-Laplacian $\mathcal{L}$ introduced in \eqref{eqn.L}  is elemental in the sub-Riemannian setting without some a priori canonical choice.

There are several fundamental issues in our construction  which are not apparent in the Riemannian setting. Such issues would prevent us from adopting a Pinsky-type process to a sub-Riemannian manifold without a reinterpretation of some standard
objects and their relations which are taken for granted in the Riemannian setting. One of these basic relations which has been exploited is the duality between the tangent and cotangent spaces. This duality is not available in the sub-Riemannian setting, which led us to the realization that it  seems to be more appropriate to construct the random walk in the cotangent space, rather than in the tangent space. This also manifests itself in the use of a compatible Riemannian metric in the definition of the sub-Laplacian $\mathcal{L}$, which allows us to overcome the problem of the non-uniqueness of solutions to the Hamilton-Jacobi equations with given initial position and velocity (tangent) vector. However, we show that the need for a compatible metric is illusory as $\mathcal{L}$ actually only depends on the corresponding ``vertical'' bundle.

We further mention that there is interest in seeing how recent work by  Bakry, Baudoin, Garafalo \emph{et al} \cite{BakryBaudoinBonnefont2009, BaudoinBonnefont2009, BaudoinBonnefont2012, BaudoinGarofalo2011} on generalized curvatures of such manifolds is related to dissipation of horizontal diffusions. We expect further study of connections between diffusions on sub-Riemannian manifolds and corresponding generators, as well as of behavior of hypoelliptic heat kernels and corresponding functional inequalities such as in \cite{BenArous1989b, DriverMelcher2005, BarilariBoscainNeel2012, LiHong-Quan2006}.

\section{Background and Notation}\label{s.nota}

\subsection{Sub-Riemannian basics}\label{S.SubRiemNota} We start by reviewing standard definitions of sub-Riemannian geometry that can be found e.g. in \cite{MontgomeryBook2002} and originally were introduced by R.~Strichartz in \cite{Strichartz1986a, Strichartz1986aCorrections}. Let $M$ be a $d$-dimensional connected smooth manifold, with tangent and cotangent bundles $TM$ and $T^*M$ respectively.

\begin{df}
For   $m \leqslant d$, let $\mathcal{H}$ be a smooth sub-bundle of $TM$ where each fiber $\mathcal{H}_q$ has dimension $m$ and is equipped with an inner product which smoothly varies between fibers. Then

\begin{enumerate}
\item the triple $\left( M, \mathcal{H}, \langle \cdot, \cdot \rangle \right)$ is called a \emph{sub-Riemannian manifold of rank} $m$;

\item $\mathcal{H}$ is called a \emph{horizontal distribution} on $M$, and  $\langle \cdot, \cdot \rangle$ a \emph{sub-Riemannian metric};

\item sections of $\mathcal{H}$ are called \emph{horizontal vector fields} and curves on $M$ whose velocity vectors are always horizontal are called \emph{horizontal curves}.
\end{enumerate}
\end{df}

\begin{as}[H\"{o}rmander's condition]
Throughout this paper we assume that the distribution $\mathcal{H}$ satisfies H\"{o}rmander's (bracket generating) condition; that is, horizontal vector fields with their Lie brackets span the tangent space $T_{q}M$ at every point $q \in M$.
\end{as}
Under H\"{o}rmander's condition any two points on $M$ can be connected by a horizontal curve by the Chow-Rachevski theorem. Thus there is a natural \emph{sub-Riemannian distance} (\emph{Carnot-Carath\'{e}odory distance}) on $M$ defined as the infimum over the lengths of horizontal curves connecting two points. In turn, this affords us the notion of a \emph{horizontal geodesic}, a horizontal curve whose length (locally) realizes the Carnot-Carath\'{e}odory distance.

 Due to degeneracy of the sub-Riemannian metric on the tangent bundle, it is convenient to introduce the \emph{cometric} on $T^{\ast}M$ corresponding to the sub-Riemannian structure. This is a particular section of the bundle of symmetric bilinear forms on the cotangent bundle,
 \[
 \llangle \cdot, \cdot \rrangle_{q}: T^{\ast}_{q}M \times T^{\ast}_{q}M \to \mathbb{R}, \ q \in M.
 \]
We relate the cometric to the sub-Riemanian metric via the \emph{sub-Riemannian bundle map} $\beta : T^{\ast}M \to TM$ with image $\mathcal{H}$ defined in the spirit of Riesz's theorem by
\begin{align}
\label{e.2.1}
\langle \beta_q(p), v \rangle_q = p(v)
\end{align}
for all $q\in M, p \in T^{\ast}_q M, \text{ and } v \in \mathcal{H}_q M.$ Hence the correspondence between the sub-Riemannian metric and cometric can be summarized as
\begin{align}
\label{e.2.2}
\llangle \varphi, \psi \rrangle_{q} = \langle \beta_{q}(\varphi), \beta_{q}(\psi) \rangle_{q} = \varphi\left( \beta_{q}(\psi) \right) = \psi \left( \beta_{q}(\varphi) \right),
\end{align}
for all $q \in M,$ and $\varphi, \psi \in T^{\ast}_{q}M$.

 Armed with the cometric, we conclude this section by defining the corresponding \emph{sub-Riemannian Hamiltonian} $H: T^{\ast} M \to \mathbb{R}$ by
 \[
   H\left( q, p \right) := \frac{1}{2} \llangle p, p \rrangle_{q}, \ q \in M, p \in T^{\ast}_{q} M
 \]
from which we can  recover the cometric via polarization. Again we note the following equivalent descriptions of the map
 \begin{equation}\label{e.2.3}
  H\left( q, p \right) = \frac{1}{2}\llangle p, p \rrangle_{q} = \frac{1}{2} \llangle \beta_{q}\left( p \right), \ \beta_{q}\left( p \right) \rrangle_{q} = \frac{1}{2} p \left( \beta_{q}\left( p \right) \right).
 \end{equation}
 The Hamiltonian is used to generate the dynamics of the system, where $H\left( q, p \right)$ gives the (kinetic) energy of a body located at $q$ with momentum $p$.

\subsection{Canonical coordinates and compatible metrics}

In the non-degenerate (Riemannian) case, the metric and cometric are matrix inverses of each other when written in any given local frame. Indeed, these matrices are represented componentwise by the lowered and raised indices $g_{ij}$ and $g^{ij}$ respectively. The degeneracy in the sub-Riemannian case disallows for such a relationship, leaving us with a choice of Riemannian metrics which will be \emph{compatible} with a given sub-Riemannian structure. The general non-canonical choice of compatible metrics will eventually lead us to defining a family of sub-Laplacians corresponding to the choice of compatible metric.

\begin{df}
Let $g$ be a Riemannian metric  on $M$ extending the sub-Riemannian metric; i.e.,  $g|_{\mathcal{H}_q \times \mathcal{H}_q} = \langle \cdot, \cdot \rangle_q$ for all $q \in M$. Then we say that $g$ is \emph{compatible} with the sub-Riemannian structure, or simply that $g$ is a \emph{compatible metric}.
\end{df}

Within this paper, the purpose of introducing a compatible metric $(\cdot, \cdot)$ is to take advantage of the induced bundle map $g : TM \to T^*M$ defined by $g(v) = ( \cdot, v)$, the standard duality $TM \leftrightarrow T^{\ast}M$ described generally through Riesz's theorem. This is a tool that we lose in the sub-Riemannian setting as we can associate to each cotangent (momentum) vector a corresponding horizontal (velocity) vector via $T^{\ast}M \stackrel{\beta}{\longrightarrow} \mathcal{H}$, but are unable to canonically map back $\mathcal{H} \stackrel{?}{\longrightarrow} T^{\ast}M$. With a compatible metric $g$ on hand, we then recover our return $\mathcal{H} \stackrel{g}{\longrightarrow} T^{\ast}M$. However,  as already mentioned, with the full strength of the Riemannian metric, we have a full bundle isomorphism $TM \to T^{\ast}M$, but this is more machinery than we need since we will only be considering the mapping on the horizontal distribution; something we explore presently through an observation from \cite{GordinaLaetsch2014b}.

\begin{prop} \label{prop.gVmatters}
Let $( \cdot, \cdot )$ be a Riemannian metric on $M$ and let $g : TM \to T^{\ast}M$ be the corresponding bundle map. Then $( \cdot, \cdot )$ is a compatible metric if and only if $\beta \circ g |_{\mathcal{H}} = \operatorname{Id}_{\mathcal{H}}$. Further, suppose $(\cdot, \cdot)_1$ and $(\cdot, \cdot)_2$ are compatible metrics with corresponding bundle maps $g_1, g_2 : TM \to T^{\ast}M$. For $i = 1,2$, let $\mathcal{V}_i$ be the orthogonal compliment of $\mathcal{H}$ in $TM$ with respect to $( \cdot, \cdot )_i$. Then $g_1(v) = g_2(v)$ for every $v \in \mathcal{H}$ if and only if $\mathcal{V}_1 = \mathcal{V}_2$.
\end{prop}

\begin{proof}[Idea of the Proof]
For a Riemannian metric $( \cdot, \cdot )$, the corresponding bundle map $g : TM \to T^{\ast}M$ can be written as $g_{\mathcal{H}} \oplus g_{\mathcal{V}} : \mathcal{H} \oplus \mathcal{V} \to T^{\ast}M$, where $\mathcal{V}$ is the $(\cdot,\cdot)$-orthogonal compliment of $\mathcal{H}$ in $TM$. From here, noticing that $g(\mathcal{V})=\Null(\beta)$, and thus $T^*M = g_{\mathcal{H}}(\mathcal{H}) \oplus \Null(\beta)$, we have $\beta = \beta_{\mathcal{H}} \oplus 0 : g_{\mathcal{H}}(\mathcal{H}) \oplus \Null(\beta) \to TM$. Moreover, $g$ is compatible if and only if $g_{\mathcal{H}} = \beta_{\mathcal{H}}^{-1}$ which in turn happens if and only if $\beta \circ g = \operatorname{Id}_{\mathcal{H}} \oplus 0$. From here, it is easy enough to deduce that the mapping $\mathcal{H} \ni v \mapsto g(v)$ depends only on $g_{\mathcal{H}}$ and $\mathcal{V}$, but not on the behavior of $g_{\mathcal{V}}$. Since, if $g$ is compatible, then $g_{\mathcal{H}} = \beta_{\mathcal{H}}^{-1}$ is completely determined by $\mathcal{V}$ and the sub-Riemannian structure, the assertions of this proposition follow.
\end{proof}

With Proposition \ref{prop.gVmatters} understood, instead of introducing a compatible metric, we could build up the remaining work by selecting a smooth \emph{vertical} sub-bundle $\mathcal{V} \subset TM$ such that $TM = \mathcal{H} \oplus \mathcal{V}$, use this to distinguish a compliment of $\Null(\beta)$, say $H$, in $T^*M$ such that we have $\beta = \beta_{\mathcal{H}} \oplus 0  : H \oplus \Null(\beta) \to TM$ and hence recover a ``return map'' with $\mathcal{H} \stackrel{\beta_{\mathcal{H}}^{-1}}{\longrightarrow} T^*M$. As for the theory that follows, the only role that  a compatible metric $g$ serves is to distinguish the vertical bundle. However, for some calculational purposes, it seems advantageous to keep working in terms of a compatible metric $g$.

\begin{nota}
\label{nota.BG}
Let $g$ be a compatible metric. For local coordinates ${\bf x} =(x^1, ..., x^d)$ on $M$, we define the local maps $\beta^{ij}: M \to \mathbb{R}$ and $g_{ij} : M \to \mathbb{R}$ by
\[
\beta^{ij}\left( q \right) := \llangle d x^i, d x^j \rrangle_{q} \text{ and } g_{ij}(q) = \Big\langle \frac{\partial}{\partial x^i}, \frac{\partial}{\partial x^j} \Big\rangle_q
\]
for all $q$ in the domain of ${\bf x}$.
The $d \times d$ matrices with entries $\beta^{ij}$ and $g_{ij}$ will be denoted by $B$ and $G$ respectively.
\end{nota}

As $B$ is the matrix representation of the bundle map $\beta : T^{\ast}M \to TM$ in local coordinates, $G$ is the local coordinate matrix representation of the bundle map $TM \to T^{\ast}M$ defined by $v \mapsto g(\cdot,v)$.

\begin{example}[Contact manifolds]\label{ex.2.1}
 Let $M$ be a $2n+1$-dimensional manifold and $\omega$ a contact $1$-form on $M$, that is, a $1$-form such that $d  \omega$ is non-degenerate on $\operatorname{Ker}(\omega)$.  Let $\mathcal{H}:=\operatorname{Ker}(\omega)$, which defines  a $n$-dimensional horizontal distribution on $M$, called a \emph{contact distribution}, and we assume that $\mathcal{H}$ is equipped with inner product $\langle \cdot, \cdot \rangle$. The sub-Riemannian manifold $\left( M, \mathcal{H}, \langle \cdot, \cdot \rangle \right)$  is called a \emph{contact sub-Riemannian manifold}. With any contact form $\omega$ we can associate its Reeb vector field, which is the unique vector field $X_{0}$ satisfying the conditions $\omega\left( X_{0} \right) = 1$ and $d \omega(X_{0}, \cdot) = 0$. Hence for any local orthonormal frame $X_{1}, ..., X_{2n}$ for the distribution $\mathcal{H}$ we have that $X_{0}, X_{1}, ..., X_{2n}$ is a local frame, since $X_{0}$ is transversal to $\mathcal{H}$. Finally, if $\langle \cdot, \cdot \rangle $ is an inner product on $\mathcal{H}$, we can extend it to $X_{0}$ by $g\left( X_{0}, X_{0} \right)=1$ and setting $\mathcal{H} \perp X_{0}$. This $g$ is then naturally compatible with the sub-Riemannian structure. Moreover, for contact sub-Riemannian manifolds there are no abnormal geodesics, that is, all geodesics are smooth and are projections of the trajectories of the Hamiltonian vector field in $T^{\ast}M$ given by the Legendre transform of the inner product on $\mathcal{H}$.  The Heisenberg group is an example of a contact manifold where $\omega$ is a standard symplectic form.
\end{example}

\subsection{Hamilton-Jacobi Equations}

We can now re-write the Hamiltonian $H : T^{\ast}M \to \mathbb{R}$ defined by \eqref{e.2.3} using canonical coordinates. By  identifying the vector $(q^1, ..., q^d, p_1, ..., p_d)$ in $\mathbb{R}^{d \times d}$ with the point $(q, p) \in T^{\ast}M$ using local coordinates for the standard identification $q^i = x^i(q)$ and $p = \sum\limits_{i=1}^{d} p_i  d  x^i$, then
\begin{equation}
\label{eqn.HamLocal}
H(q,p) = \frac{1}{2} \sum_{i,j = 1}^{d} p_i p_j \beta^{ij}(q).
\end{equation}
A curve $(q(t), p(t)) \in T^{\ast}M$ satisfies the \emph{Hamilton-Jacobi equations} when
\begin{align}
&\dot{q}^i(t) = \frac{\partial H}{\partial p_i}\left( q(t), p(t)\right) = \frac{1}{2} \sum_{j =1}^{d} p_j(t) \beta^{ij}(q(t))
\label{eqn.HJinCC1}
\\
&\dot{p}_i(t) = -  \frac{\partial H}{\partial q^i}\left(q(t), p(t) \right) = \sum_{k, j=1}^{d} p_k(t) p_j(t) \frac{\partial \beta^{kj}}{ \partial q^i} \left|_{q(t)}\right.
\label{eqn.HJinCC2}
\end{align}
where we have slightly abused notation in the common way, conflating $\frac{\partial}{\partial p_i}$ with the partial derivative of \eqref{eqn.HamLocal} in terms of $p_i$, and  $\frac{\partial}{\partial q^i}$ with $\frac{\partial}{\partial x^i}$ in \eqref{eqn.HJinCC2}. Equations \eqref{eqn.HJinCC1} and \eqref{eqn.HJinCC2} are collectively known as the Hamilton-Jacobi equations.

Taking a time derivative in \eqref{eqn.HJinCC1}  we get
\begin{equation}
\label{eqn.qddot1}
{\ddot{q}}^k(t) = \sum_{i, j, l = 1}^{d} \left\{ \beta^{il}(q(t)) \frac{\partial \beta^{kj}}{\partial q^l}\left|_{q(t)}\right. - \frac{1}{2}  \beta^{kl}(q(t))  \frac{\partial \beta^{ij}}{\partial q^l} \left|_{q(t)}\right.\right\}p_i(t) p_j(t).
\end{equation}
Define the \emph{raised Christoffel symbols} locally by
\begin{equation}
\label{eqn.raisedChris}
\Gamma^{ijk}(q) := - \frac{1}{2} \sum_{l = 1}^{d} \left\{ \beta^{il}(q) \frac{\partial \beta^{jk}}{\partial x^l}\Big|_{q} + \beta^{jl}(q) \frac{\partial \beta^{ik}}{\partial x^l} \Big|_q-   \beta^{lk}(q) \frac{\partial \beta^{ij}}{\partial x^l} \Big|_{q}\right\}.
\end{equation}
Rewriting \eqref{eqn.qddot1} with \eqref{eqn.raisedChris} while suppressing the time dependence,

\begin{equation}
\label{eqn.qddot2}
{\ddot{q}}^k = - \sum_{i, j = 1}^{d} \Gamma^{i j k}(q) p_i p_j.
\end{equation}
The negative signs in \eqref{eqn.raisedChris} and \eqref{eqn.qddot2} are just by convention so that the acceleration term is consistent with standard Riemannian definitions.

\begin{nota}\label{n.3.1}
We let $\Phi$ be the flow of the Hamilton-Jacobi equations \eqref{eqn.HJinCC1} and \eqref{eqn.HJinCC2}. That is, $\Phi$ is a map
\[
\Phi : [0, \tau) \times T^{\ast}M  \longrightarrow T^{\ast}M,\]
 such that if $(x,p) \in T^{\ast}_xM$ then $t \mapsto \Phi_t(x,p)$ is the curve $(q(t), p(t))$ in $T^{\ast}M$ satisfying the Hamilton-Jacobi equations with initial conditions $q(0) = x$ and $p(0) = p$ for $t$ in some maximal interval $[0,\tau)$.
\end{nota}

\begin{rem}
If $(q(t),p(t)) = \Phi_t(x,p)$, then $q(t)$ is a horizontal curve. Indeed, \eqref{eqn.HJinCC1} gaurantees that $\dot{q}(t) = \beta(p(t)) \in \mathcal{H}_{q(t)}$.
\end{rem}

\section{Horizontal sub-Laplacians and the Heisenberg group}
\label{section.subLaplace}
In this section we introduce a family of second order differential operators on $M$ indexed by Riemannian metrics compatible with the sub-Riemannian structure $\left( M, \mathcal{H}, \langle \cdot, \cdot \rangle \right)$. In the Riemannian case when $\mathcal{H} = TM$, we recover the Laplace-Beltrami up to a constant scaling factor; in the Heisenberg case using the standard compatible metric introduced in Example \ref{ex.2.1}, we get the familiar sums of squares Laplacian up to a constant scaling factor.

\subsection{Horizontal sub-Laplacians}
\label{ssection.subLaplace}

Definition \ref{defi.Lg} below introduces horizontal sub-Laplacian operators, but before we can give the definition, some notation is in order.

\begin{nota}
We denote the unit sphere in $\mathcal{H}_x$ by $\mathcal{S}^{\mathcal{H}}_x := \{ v \in \mathcal{H}_x : \langle v, v \rangle_x = 1\}$. The (unique) rotationally invariant measure on $\mathcal{S}_x$ will be denoted $\mathbb{U}_x$.
\end{nota}

\begin{df} \label{defi.Lg}
Let $(\cdot, \cdot)$ be a compatible metric, and let $g$ be the corresponding bundle map $TM \to T^{\ast}M$. We define $\mathcal{L} : C^{\infty}_c(M) \to \mathbb{R}$ as
\begin{equation}\label{eqn.L}
\mathcal{L} f (x) := \int_{\mathcal{S}^{\mathcal{H}}_x} \left\{ \frac{d}{dt} \Big|_0 \frac{d}{ds} \Big|_0 f \big(\Phi_{t+s}(x,g(v))\big) \right\} \mathbb{U}_x(dv).
\end{equation}
We will call $\mathcal{L}$ the \emph{horizontal sub-Laplacian} corresponding to $g$.
\end{df}

As is now obvious from Proposition \ref{prop.gVmatters} and the remarks that followed, we have the following statement.

\begin{prop} \label{prop.LonlyonV}
Suppose $(\cdot, \cdot)_1$ and $(\cdot, \cdot)_2$ are compatible metrics giving rise to orthogonal compliments $\mathcal{V}_1$ and $\mathcal{V}_2$ of  $\mathcal{H}$, respectively. For $i = 1,2$, if $\mathcal{L}_i$ is defined by \eqref{eqn.L} with respect to $(\cdot, \cdot)_i$, then $\mathcal{L}_1 = \mathcal{L}_2$ whenever $\mathcal{V}_1 = \mathcal{V}_2$.
\end{prop}

\subsection{A formula for $\mathcal{L}$ in local coordinates}
Working in local coordinates, we set $q(t) := \pi\left( \Phi_t(x, p) \right)$, where $\pi$ is the projection onto $M$. Defining $v = \beta(p)$, we get

\begin{equation}
\label{eqn.2ndDeriv}
\begin{aligned}
&\frac{d}{dt} \Big|_0 \frac{d}{ds}\Big|_0 f \left(q(t+s)\right) = \frac{d}{dt}
\Big|_0  \left\{ \sum_{i=1}^{d} \dot{q}^i(t) \frac{\partial f}{\partial x^i} \Big|_{q(t)} \right\}
\\
&= \sum_{i=1}^{d}\left\{ {\ddot{q}}^i(0) \frac{\partial f}{\partial x^i}\Big|_{x} + \sum_{j=1}^{d} \dot{q}^i(0) \dot{q}^j(0) \frac{\partial^2 f}{\partial x^i  \partial x^j}\Big|_x\right\}
\\
&= \sum_{i=1}^{d} \left\{ -\sum_{k,l=1}^{d} \Gamma^{k l i}(x)\, p_k p_l  \frac{\partial f}{\partial x^i} \Big|_x + \sum_{j=1}^{d} v^i v^j \frac{\partial^2 f}{\partial x^i  \partial x^j} \Big|_x \right\}
\\
&=  \sum_{i, j =1}^{d} \left\{ v^i v^j \, \frac{\partial^2 f}{\partial x^i  \partial x^j} \Big|_x - \sum_{k=1}^{d} \Gamma^{i j k}(x) p_i \, p_j  \frac{\partial f}{\partial x^k} \Big|_x \right\}
\end{aligned}
\end{equation}

\begin{prop} \label{prop.unifint}
Let $(\cdot, \cdot)$ be a compatible metric with corresponding bundle map $g : TM \to T^{\ast}M$. For $1 \leqslant i, j \leqslant d$,
\begin{align}
\int_{\mathcal{S}^{\mathcal{H}}_x} v^i v^j  \mathbb{U}_x(dv) = \frac{1}{m}  \beta^{ij}(x)
\label{eqn.unifint1}
\end{align}
and
\begin{align}
\int_{\mathcal{S}^{\mathcal{H}}_x} p_i  p_j  \mathbb{U}_x(dv) = \frac{1}{m} \sum_{a, b=1}^{d} g_{i a}  \beta^{ab} g_{bj}(x).
 \label{eqn.unifint2}
\end{align}
Here $p = g(v)$.
\end{prop}

\begin{proof}
Rewrite \eqref{eqn.unifint1} as
\begin{align*}
\int_{\mathcal{S}^{\mathcal{H}}_x} dx^i (v) dx^j(v)  \mathbb{U}_x(dv) = \int_{\mathcal{S}^{\mathcal{H}}_x} \langle \beta_x(dx^i), v \rangle \langle\beta_x(dx^j), v \rangle  d\mathbb{U}_x(v)
\\
= \frac{1}{m} \langle \beta_x(dx^i), \beta_x(dx^j) \rangle =  \frac{1}{m}\beta^{ij}(x)
\end{align*}
The second equality follows from Corollary \ref{cor.equalint} below. From here \eqref{eqn.unifint2} follows by a similar argument after  realization that $p_i = \sum\limits_{a =1}^{d} g_{ia} v^a$ and $p_j = \sum\limits_{b=1}^{d} g_{jb} v^b$.
\end{proof}
Combining Proposition \ref{prop.unifint} with \eqref{eqn.2ndDeriv} leads immediately to

\begin{thm}
The horizontal sub-Laplacian indexed by $g$ can be locally written as
\begin{equation}
\begin{aligned}
\mathcal{L} = \frac{1}{m} \sum_{i,j = 1}^{d} \left\{ \beta^{ij} \frac{\partial^2 }{\partial x^i  \partial x^j}  - \sum_{a,b,k=1}^{d} \Gamma^{i j k} g_{ia}\beta^{ab} g_{bj}  \frac{\partial }{\partial x^k} \right\}
\\
= \frac{1}{m} \sum_{i,j=1}^{d} \left\{ \beta^{ij} \frac{\partial^2 }{\partial x^i  \partial x^j}  - \sum_{k=1}^{d} \Gamma^{i j k} \left[ GBG\right]_{ij}   \frac{\partial }{\partial x^k} \right\}
\label{eqn.Lglocal}
\end{aligned}
\end{equation}
where $\left[ GBG \right]_{ij}$ is the $ij$th entry of the matrix $GBG$ and $G$ and $B$ are defined in Notation \ref{nota.BG}.
\end{thm}

\begin{rem}
In the case that $\mathcal{H} = TM$, $B = G^{-1}$ and hence
\[
\mathcal{L} = \frac{1}{m} \sum_{i,j = 1}^{d} \left\{ \beta^{ij} \frac{\partial^2 }{\partial x^i  \partial x^j}  - \sum_{k=1}^{d} \Gamma^{i j k} g_{ij}  \frac{\partial }{\partial x^k} \right\},
\]
which is the ($\frac{1}{m}$ scaled) local formula for the Laplace-Beltrami operator on the Riemannian manifold $(M, g)$.
\end{rem}

With Proposition \ref{prop.LonlyonV} in mind, \eqref{eqn.Lglocal} appears deceivingly dependent on the structure of the compatible metric with the repeat appearance of its corresponding matrix $G$. However, using the notation in the proof of Proposition \ref{prop.gVmatters}, we have $g \circ \beta \circ g = g_{\mathcal{H}} \circ 0$, which as the proof of and remarks following Proposition \ref{prop.gVmatters} indicate, $g_\mathcal{H}$ is determined by the sub-Riemannian structure once the vertical bundle $\mathcal{V}$ is fixed. The following example in the Heisenberg case illustrates this.

\subsection{An example: the Heisenberg group}
\label{ssection.Heisen}
Let $\mathbb{H}$ be the Heisenberg group; that is, $\mathbb{H}\cong\mathbb{R}^3$ with the multiplication defined by

\[
\left( x_{1}, y_{1}, z_{1} \right) \star \left( x_{2}, y_{2}, z_{2} \right):=\left( x_{1}+x_{2}, y_{1}+y_{2}, z_{1}+z_{2} + \frac{1}{2}  \omega\left( x_{1}, y_{1}; x_{2}, y_{2}\right)\right),
\]
where $\omega$ is the standard symplectic form

 \[
 \omega\left( x_{1}, y_{1}; x_{2}, y_{2}\right):=x_{1}y_{2} - y_{1}x_{2}.
 \]
 Left multiplication by $\left( x, 0, 0 \right)$ and $\left( 0, y, 0 \right)$ induce two left-invariant vector fields

\begin{align}
X\left( q \right) := \frac{\partial}{\partial x}\Big|_{q} - \frac{1}{2} y \frac{\partial}{\partial z}\Big|_{q}
\notag
\\
Y\left( q \right) := \frac{\partial}{\partial y}\Big|_{q} + \frac{1}{2} x \frac{\partial}{\partial z}\Big|_{q}
 \label{eqn.XYheis}
\end{align}
for any $q \in \mathbb{H}$. At each point $q \in \mathbb{H}$ the globally defined vector fields $X\left( q \right)$ and $Y\left( q \right)$ span a two-dimensional subspace of $T_{q}\mathbb{H}$; set $\mathcal{H}_{q}:=\operatorname{Span} \left\{ X\left( q \right),  Y\left( q \right) \right\}$ and then

\[
\mathcal{H}:=\bigcup_{q \in \mathbb{H}} \mathcal{H}_{q}
\]
can be taken as the horizontal distribution. Moreover, at each $q \in \mathbb{H}$ we have

\[
[ X\left( q \right), Y\left( q \right)]=\frac{\partial}{\partial z}\Big|_{q}=: Z\left( q \right),
\]
and so H\"{o}rmander's condition is satisfied. Consider $M=\mathbb{H}$, the horizontal distribution $\mathcal{H}$ defined as above, and the inner product $\langle \cdot, \cdot \rangle$ on $\mathcal{H}_{q}$ defined so that $\{ X\left( q \right), Y\left( q \right) \}$ is an orthonormal basis for $\mathcal{H}_{q}$. Recall also that in Example \ref{ex.2.1} we described $\left( \mathbb{H}, \mathcal{H}, \langle \cdot, \cdot \rangle \right)$  as a contact manifold with $Z$ as a Reeb vector field.

A covector $\varphi \in T_p^{\ast}M$ will be identified with the triple $(\varphi_1, \varphi_2, \varphi_3) \in \mathbb{R}^3$ via $\varphi = \varphi_1 dx + \varphi_2  dy + \varphi_3 dz$.  We have
that for each $q = (x,y,z) \in \mathbb{H}$, the sub-Riemannian bundle map $\beta : T^{\ast}M \to TM$ is defined by
\begin{align}
 (\varphi_1, \varphi_2, \varphi_3) \xrightarrow{\beta_{q}}
\left( \varphi_1 - \frac{1}{2}y \varphi_3,
\varphi_2 + \frac{1}{2}x \varphi_3,
\frac{1}{2}(x\varphi_2 - y \varphi_1) + \frac{1}{4}(y^2 + x^2)\varphi_3
\right).
\end{align}
\label{eqn.HeisBeta}
The matrix representation of $\beta$ with entries $\beta^{ij} = dx^i(\beta(dx^j))$ is
\begin{equation} \label{eqn.HeisB}
B(x,y,z) =
\begin{pmatrix}
1 & 0 & -\frac{y}{2} \\
0 & 1 & \frac{x}{2} \\
-\frac{y}{2} & \frac{x}{2} & \frac{x^2 + y^2}{4}
\end{pmatrix}
\end{equation}

Using the fact that $\left( \mathbb{H}, \mathcal{H}, \langle \cdot, \cdot \rangle \right)$ is a contact sub-Riemannian manifold, we can extend the sub-Riemannian metric $\langle \cdot, \cdot \rangle$ to the Riemannian metric $g$ which makes $\{X,Y,Z\}$ a global orthogonal frame with $g(Z,Z) = \lambda > 0$.
The matrix representation of $g$ with entries $g_{ij} = g( \frac{\partial}{\partial x^j}, \frac{\partial}{\partial x^i})$ is
\begin{equation} \label{eqn.HeisG}
G(x,y,z) =
\begin{pmatrix}
1+ \frac{\lambda y^2}{4} & -\frac{\lambda xy}{4} & \frac{\lambda y}{2} \\
-\frac{\lambda xy}{4} & 1 + \frac{\lambda x^2}{4} & -\frac{\lambda x}{2} \\
\frac{\lambda y}{2} & -\frac{\lambda x}{2} & \lambda
\end{pmatrix},
\end{equation}
and therefore,
\begin{align*}
 GBG = \begin{pmatrix}
1+ \frac{\lambda y^2}{4} & -\frac{\lambda xy}{4} & \frac{\lambda y}{2} \\
-\frac{\lambda xy}{4} & 1 + \frac{\lambda x^2 }{4} & -\frac{\lambda x}{2} \\
\frac{\lambda y}{2} & -\frac{\lambda x}{2} & \lambda
\end{pmatrix}
 \begin{pmatrix}
1 & 0 & -\frac{y}{2}\\
0 & 1 & \frac{x}{2}\\
-\frac{y}{2} & \frac{x}{2} & \frac{x^2 + y^2}{4}
\end{pmatrix}
 \begin{pmatrix}
1+ \frac{\lambda y^2}{4} & -\frac{\lambda xy}{4} & \frac{\lambda y}{2} \\
-\frac{\lambda xy}{4} & 1 + \frac{\lambda x^2}{4} & -\frac{\lambda x}{2} \\
\frac{\lambda y}{2} & -\frac{\lambda x}{2} & \lambda
\end{pmatrix} \\
  =
\begin{pmatrix}
1 & 0 & 0 \\
0 & 1 & 0\\
0 & 0 & 0
\end{pmatrix}.
\end{align*}

 Here you can see the manifestation of Proposition \ref{prop.LonlyonV} through the independence of $GBG$ on any choice of $\lambda$.
Using \eqref{eqn.raisedChris} and \eqref{eqn.HeisB}, for any $k=1,2,3$, $\Gamma^{11k} = \Gamma^{22k} = 0$, which gives us all values needed to explicitly find \eqref{eqn.Lglocal} in this context.
\begin{align*}
\mathcal{L} &= \frac{1}{2} \sum_{i, j =1}^{3} \left\{ \beta^{ij}  \frac{\partial^2}{\partial x^i \partial x^j} \right\} -\frac{1}{2} \sum_{i, j, k=1}^{3} \left\{ \Gamma^{i j k} \left[ GBG \right]_{ij}  \frac{\partial }{\partial x^k} \right\} \\
&= \frac{1}{2} \sum_{i, j =1}^{3} \Bigg\{ \beta^{ij}  \frac{\partial^2}{\partial x^i \partial x^j} \Bigg\} - 0\\
&= \frac{1}{2}\Bigg\{ \frac{\partial^2}{\partial x^2} + \frac{\partial^2}{\partial y^2} + \frac{1}{4} (x^2 + y^2) \frac{\partial^2}{\partial z^2} - y \frac{\partial^2}{\partial x \partial z} + x \frac{\partial^2}{\partial y \partial z} \Bigg\}
\end{align*}
Thus we can rewrite $\mathcal{L}$ as
\[
\mathcal{L}  = \frac{1}{2}\left(X^2 + Y^2 \right).
\]

\section{Weak Convergence and Random Walks}\label{s4}

The first part of this section discusses the weak convergence results necessary to prove the convergence of the random walk developed in Section \ref{ssection.RW} to a horizontal Brownian motion. The main result is Theorem \ref{thm.RW}.

\subsection{Convergence of semigroups}

Let $C_c^{\infty}(M)$ and $C_c^{\infty}(T^{\ast}M)$ be the spaces of the smooth, compactly supported real-valued functions on $M$ and $T^{\ast}M$ equipped with the $\sup$ norm. We identify $C_c^{\infty}(M)$ with a closed subspace of $C_c^{\infty}(T^{\ast}M)$: if $f \in C_c^{\infty}(M)$ then the element $\tilde{f} \in C_c^{\infty}(T^{\ast}M)$ identified with $f$ is given by $\tilde{f}(x,p) := f(x)$.

\begin{df}\label{defi.DHJ}

For $f \in C_c^{\infty}(T^{\ast}M)$, the \emph{Hamilton-Jacobi flow field}, $\mathscr{D}_{HJ} : C_c^{\infty}(T^{\ast}M) \to C_c^{\infty}(T^{\ast}M)$, is defined by
\begin{equation}
\mathscr{D}_{HJ} f(x, p) = \frac{d}{dt} \Big|_{t=0} f \left( \Phi_t(x, p)\right). \label{eqn.DHJ} \end{equation}
\end{df}

\begin{rem}
If $f \in C_c^{\infty}(M)$, then $\mathscr{D}_{HJ} f(x,p) = v(f)$ where $v = \beta(p)$.
\end{rem}

\begin{rem} The semigroup property of flows implies that if $f \in C_c^{\infty}(T^{\ast}M)$, then
\begin{align}
\label{eqn.DHJ2}
\mathscr{D}_{HJ} \left( \mathscr{D}_{HJ} f \right)(x, p) = \frac{d}{ds}\Big|_{s=0} \frac{d}{dt} \Big|_{t=0} f\left( \Phi_{t+s}(x,p) \right).
\end{align}
\end{rem}

\begin{df}
\label{defi.UAP}
For $f \in C_c^{\infty}(TM)$, the \emph{horizontally averaged projection}

\[\mathcal{P} : C_c^{\infty}(T^{\ast} M) \to C_c^{\infty}(M),\]
is defined by
\begin{equation}
\label{eqn.UAP}
\mathcal{P} f(x) = \int_{\mathcal{S}^{\mathcal{H}}_x} f (x, g(v))  \mathbb{U}_x(dv).
\end{equation}
Here, as before, $\mathbb{U}_x$ is the rotationally invariant (uniform) probability measure on the unit sphere $\mathcal{S}^{\mathcal{H}}_x$.
\end{df}
Let us now make the following observation.
\begin{prop} \label{prop.PDDisLg}
For every $f \in C_c^{\infty}(M)$,
$\mathcal{L}f = \mathcal{P} \mathscr{D}_{HJ} \mathscr{D}_{HJ} f$.
\end{prop}

 Set $\mathcal{I}$ as the identity operator on $C_c^{\infty}(T^{\ast}M)$. We denote by $e^{t(\mathcal{P} - \mathcal{I})}$ the strongly continuous contraction semigroup on $C_c^{\infty}(T^{\ast}M)$ whose bounded generator is $\mathcal{P} - \mathcal{I}$.
 We denote by  $e^{t \mathscr{D}_{HJ}}$ the  strongly continuous contraction semigroup on $C_c^{\infty}(T^{\ast}M)$ whose generator is $\mathscr{D}_{HJ} $. Using Notation \ref{n.3.1} we have
 \[ e^{t \mathscr{D}_{HJ}} f(x, p)  = f\left(\Phi_t(x, p)\right). \]
 Finally, for any $\alpha > 0$ we denote by $T_{\alpha}(t)$ the strongly continuous contraction semigroup on $C_c^{\infty}(T^{\ast}M)$ whose generator is $\mathscr{D}_{HJ} + \alpha(\mathcal{P}-\mathcal{I})$. This is possible since $\mathcal{P}-\mathcal{I}$ is bounded.

 For more generalized notions of summing together generators we refer to \cite{Trotter1959a}. Our aim is to prove a limit theorem of $T_{\alpha}(\alpha t)$ as $\alpha \to \infty$ using \cite[Theorem 2.2]{Kurtz1973a}. To this end, we first state some prerequisites which follow easily from the definitions.

\begin{lem}\label{lem.KurtzPrereq1}
The following hold.

1) $\operatorname{Ran}(\mathcal{P}) = C_c^{\infty}(M)$.

2) $\mathcal{P} \mathscr{D}_{HJ} f = 0$ for $f \in C_c^{\infty}(M)$.

\end{lem}
Following the notation of T.~Kurtz in \cite{Kurtz1973a}, define
\begin{equation} \label{eqn.KurtzD0}
\begin{aligned}
 D_0  := & \left\{f \in \operatorname{Dom}(\mathscr{D}_{HJ}) \cap \operatorname{Ran}(\mathcal{P}):\right.
\\
& \left. \text{ there exists } h \in \operatorname{Dom}(\mathscr{D}_{HJ}) \text{ such that } (\mathcal{P}-\mathcal{I})h = - \mathscr{D}_{HJ} f \right\}.
\end{aligned}
\end{equation}
Using the first claim of Lemma \ref{lem.KurtzPrereq1}, we see that $\operatorname{Dom}(\mathscr{D}_{HJ}) \cap \operatorname{Ran}(\mathcal{P}) = C_c^{\infty}(M)$. Moreover, for $f \in C_c^{\infty}(M)$, define $h := \mathscr{D}_{HJ} f$. By the second claim of Lemma \ref{lem.KurtzPrereq1},  $(\mathcal{P}-\mathcal{I}) h = -\mathscr{D}_{HJ} f$. We conclude that $D_0 = C_c^{\infty}(M)$.

Before getting to Theorem \ref{thm.weakConv}, the main result regarding weak convergence to a sub-Riemannian Brownian motion, we first make an assumption necessary to apply the result of Kurtz we wish to use.

\begin{assumption}
\label{assume.feller}
We henceforth assume the semigroup $e^{t \mathcal{L} }$ is Feller in the following sense: for every $t \geqslant 0$, $\lambda > 0$, and $h \in C^{\infty}_c(M)$

\[
x \longmapsto \int_0^{\infty} e^{-\lambda t} e^{t \mathcal{L} }h(x) dt \in C^{\infty}_c(M).
\]
\end{assumption}
We can now formulate the main result of this section.
\begin{thm}
\label{thm.weakConv}
For every $f \in C^{\infty}_c(M)$,
\[
 \lim_{\alpha \to \infty} T_{\alpha}(\alpha t)f = e^{t \mathcal{L}} f,
\]
where the limit is taken in the $\sup$ norm.
\end{thm}

\begin{proof}
By Assumption \ref{assume.feller}, for any $h \in C^{\infty}_c(M) = D_0$ and $\lambda > 0$, the function $k(x) = \int_0^{\infty} e^{-\lambda t} e^{t \mathcal{L}}h(x) dt$ is in $C_c^{\infty}(M)$; moreover, $(\lambda - \mathcal{L}) k = h$. This shows that $C^{\infty}_c(M) \subset \operatorname{Ran}(\lambda - \Delta_{\mathcal{H}})$. Hence by \cite[Theorem 2.2]{Kurtz1973a}, the closure of $\mathcal{P} \mathscr{D}_{HJ} \mathscr{D}_{HJ}$ is the generator of a strongly continuous contraction semigroup $e^{t \mathcal{P} \mathscr{D}_{HJ} \mathscr{D}_{HJ}}$ such that
\[
\lim_{\alpha \to \infty} T_{\alpha}(\alpha t) f = e^{t \mathcal{P} \mathscr{D}_{HJ} \mathscr{D}_{HJ}} f
\]
for every $f \in C_c^{\infty}(M)$, where the limit is in the $\sup$ norm. As noted in Proposition \ref{prop.PDDisLg}, $\mathcal{P} \mathscr{D}_{HJ} \mathscr{D}_{HJ} = \mathcal{L}$ on $C^{\infty}_c(M)$. This concludes the proof.
\end{proof}

\subsection{A sub-Riemannian random walk}
\label{ssection.RW}

\begin{assumption} \label{assumption.HopfRinow}
We assume that the sub-Riemannian manifold $\left( M, \mathcal{H}, \langle \cdot, \cdot, \rangle \right)$ is complete with respect to the Carnot-Carath\'{e}odory metric. Note that in this case this sub-Riemannian manifold is also geodesically complete, that is, all geodesics are defined for all $t \geqslant 0$ by a sub-Riemannian Hopf-Rinow theorem (e.g. \cite[Theorem 7.1]{Strichartz1986a}).
\end{assumption}

Let $\varepsilon > 0$ be a parameter that we eventually take to zero. Let $\{e_i\}_{i=1}^{\infty}$ be i.i.d. exponential random variables with parameter 1 and define $e_0 := 0$. Let us fix $(x,p) \in T^*M$ as our initial position and momentum and let $v = \beta(p)$. Define $(\xi^{\varepsilon}_t, p^{\varepsilon}_t) = \Phi_{\varepsilon t}(x,g(v))$ for $0 \leq t < e_1$. Given $e_1$, let $x_1^{\varepsilon} = \pi \circ \Phi_{\varepsilon e_1}(x,g(v)) \in T^{\ast}M$ where $\pi :T^{\ast} M \to M$ is the canonical projection, and take $v^{\varepsilon}_1$ randomly from $\mathcal{S}_{x^{\varepsilon}_1}^{\mathcal{H}}$ such that the law of $v^{\varepsilon}_1$ is $\mathbb{U}_{x^{\varepsilon}_1}$. From here, for $e_1 \leq t < e_2$, define $(\xi^{\varepsilon}_t, p^{\varepsilon}_t) = \Phi_{\varepsilon (t - e_1)} (x_1^{\varepsilon}, g(v_1^{\varepsilon}))$. Continuing recursively, for each $k\geq 0$,  once given $\{(x_0,v_0), (x_1^{\varepsilon}, v_1^{\varepsilon}), ..., (x_k^{\varepsilon}, v_k^{\varepsilon})\}$ and $\{e_i\}_{i=1}^{k+1}$, define $x_{k+1}^{\varepsilon} = \pi\big(\Phi_{\varepsilon e_{k+1}}(x_k^{\varepsilon}, g(v_k^{\varepsilon}))\big)$ and take $v^{\varepsilon}_{k+1}$ randomly from $\mathcal{S}_{x^{\varepsilon}_{k+1}}^{\mathcal{H}}$ such that the law of $v^{\varepsilon}_{k+1}$ is $\mathbb{U}_{x^{\varepsilon}_{k+1}}$. From here, for $e_{k+1} \leq t < e_{k+2}$ define $(\xi^{\varepsilon}_t, p^{\varepsilon}_t) = \Phi_{\varepsilon (t - e_{k+1})} (x_{k+1}^{\varepsilon},g( v_{k+1}^{\varepsilon}))$.

We now have a ($\varepsilon$-scaled) random walk $B_t^{\varepsilon}(x,p) := (\xi_t^{\varepsilon}, p_t^{\varepsilon} )$ in $T^* M$. Here, the notation $B_t^{\varepsilon}(x,p)$ emphasizes that $(x,p)$ are the initial conditions (and $\beta(p) = v$ is the initial horizontal velocity). Define $T_t^{\varepsilon} : C^{\infty}_c(T^*M) \to C^{\infty}_c(T^*M)$ by
\[
T_t^{\varepsilon}f(x,p) = \mathbb{E}\left[ f(B_t^{\varepsilon}(x,p))\right].
\]
With this we are ready to present the final piece needed, Theorem \ref{thm.SemiGroupTt}, before the statement of convergence, Theorem \ref{thm.RW}. Our setup to this point is such that we can use a weaker version  of the argument in \cite[Proposition 3.3]{Pinsky1976a} to prove Theorem \ref{thm.SemiGroupTt}. The inability to reproduce the stronger statement arises from the fact that $B_t(x,p_1) = B_t(x,p_2)$ when $\beta(p_1) = \beta(p_2)$, even though $\Phi_t(x,p_1)$ need not be equal to $\Phi_t(x,p_2)$.

\begin{thm}
\label{thm.SemiGroupTt}
For every $f \in C^{\infty}_c(M)$,
\begin{align}
T^{\varepsilon}_t f = e^{t(\varepsilon \mathscr{D}_{HJ} + \mathcal{P}-\mathcal{I})}f.
\end{align}
\end{thm}

Before exposing the proof of Theorem \ref{thm.SemiGroupTt} (which is given below in Section \ref{app.thm.SemiGroupTt}), let us note that as a corollary, we arrive at the convergence result which is our main theorem.
\begin{thm}
\label{thm.RW}
For every $f \in C_c^{\infty}(M)$,
\begin{align}
\lim_{\varepsilon \to 0} T^{\varepsilon}_{t/\varepsilon^2} f = e^{t \mathcal{L}} f.
\end{align}
\end{thm}
\begin{proof}
From Theorem \ref{thm.weakConv}, it follows that if $f \in C^{\infty}_c(M)$ that $\lim\limits_{\varepsilon \to 0} e^{(t/\varepsilon^2)(\varepsilon \mathscr{D}_{HJ} + \mathcal{P}-\mathcal{I})}f =e^{t \mathcal{L}} f$. Since Theorem \ref{thm.SemiGroupTt} shows that $T^{\varepsilon}_t$ and $e^{t(\varepsilon \mathscr{D}_{HJ} + \mathcal{P}-\mathcal{I})}$ agree on $C^{\infty}_c(M)$, the result follows.
\end{proof}

\section{Proofs}

\subsection{The Proof of Theorem \ref{thm.SemiGroupTt}}
 \label{app.thm.SemiGroupTt}
We continue with the notation introduced in Section \ref{ssection.RW}. For the i.i.d.~exponential random variables $\{e_i\}_{i=1}^{\infty}$ and for $k \geq 0$, let $\tau_k = e_0 + e_1 + \cdots + e_k$; recall that $e_0 := 0$. We denote by $R^{\varepsilon}_{\lambda}$ the resolvent of $e^{\varepsilon t \mathscr{D}_{HJ}}$; that is,
\[
R^{\varepsilon}_{\lambda}f(x,p) =e^{\varepsilon t \mathscr{D}_{HJ}}f(x,p) =  f( \Phi_{\varepsilon t}(x,p)).
\]
We denote by $S^{\varepsilon}_{\lambda}$ the resolvent of $T^{\varepsilon}_t$; that is,
\begin{align}
S^{\varepsilon}_{\lambda} f(x, p) = \int_0^{\infty} e^{-\lambda t} \mathbb{E} \left[ f(B^{\varepsilon}_t(x,p))\right] dt.
\end{align}

\begin{lem}
\label{lem.PrfSemiGrp1}
For any $f \in C_c^{\infty}(T^*M)$,
\begin{align*}
\mathbb{E}\bigg[ \int_0^{ \tau_1} e^{-\lambda t} f(B_t^{\varepsilon}(x,p)) \,dt \bigg] = R^{\varepsilon}_{1+\lambda} f(x,g \circ \beta(p)).
\end{align*}
\end{lem}
\begin{proof}
If the initial conditions of $B_t^{\varepsilon}$ are $(x,p)$, then or $0 \leqslant t < \tau_1$, $B_t^{\varepsilon} =\Phi_{\varepsilon t}(x,g \circ \beta(p))$. Thusly
\begin{align*}
&\mathbb{E}_{(x,p)}\bigg[ \int_0^{\tau_1} e^{-\lambda t} f(B^{\varepsilon}_t) dt \bigg] = \mathbb{E}_{(x,p)}\bigg[ \int_0^{\tau_1} e^{-\lambda t} f(\Phi_{\varepsilon t}(x,g\circ \beta(p))) dt \bigg]  \\
&=  \int_0^{\infty}\int_0^t e^{- s}e^{-\lambda t}  f(\Phi_{\varepsilon t}(x,g\circ \beta(p)))\, dt\, ds = \int_0^{\infty}e^{-(\lambda + 1)t} f(\Phi_{\varepsilon t}(x,g \circ \beta(p)))\, dt\\
& = R^{\varepsilon}_{1+\lambda} f(x,g \circ \beta(p)).
\end{align*}
This concludes the proof.
\end{proof}

\begin{lem}
\label{lem.PrfSemiGrp2}
For any $f \in C_c^{\infty}(T^*M)$,
\begin{align*}
\mathbb{E}\bigg[  \int_{\tau_1}^{\infty} e^{-\lambda t} f(B^{\varepsilon}_t(x,p)) dt \bigg] = R^{\varepsilon}_{1+\lambda} \mathcal{P} S^{\varepsilon}_{\lambda} f(x, g \circ \beta(p)).
\end{align*}
\end{lem}
\begin{proof} Notice that
\[
\mathbb{E}\bigg[  \int_{\tau_1}^{\infty} e^{-\lambda t} f(B^{\varepsilon}_t(x,p)) dt \bigg] = \mathbb{E}\bigg[ e^{- \lambda \tau_1}  \int_{0}^{\infty} e^{-\lambda t} f(B_t^{\varepsilon}(x_1^{\varepsilon},g(v_1^{\varepsilon}))) dt \bigg],
\]
\[
\mathbb{E}\bigg[  \int_{0}^{\infty} e^{-\lambda t} f(B_t^{\varepsilon}(x_1^{\varepsilon},g(v_1^{\varepsilon}))) dt \ \Big| \ (x_1^{\varepsilon}, v_1^{\varepsilon} ) \bigg] = S^{\varepsilon}_{\lambda} f(x_1^{\varepsilon}, g(v^{\varepsilon}_1)),
\]
and
\begin{align*}
&\mathbb{E}\left[ S_{\lambda}^{\varepsilon} f(x_1^{\varepsilon}, g(v_1^{\varepsilon})) \ \big| \ \tau_1= t \right] = \mathbb{E}\left[   S_{\lambda}^{\varepsilon} f(x_t, g(U)) \right] \\
&\qquad = \int_{\mathcal{S}^{\mathcal{H}}_{x_t}} S_{\lambda}^{\varepsilon} f(x_t, g(v)) \,\mathbb{U}_{x_t}(dv)  = \mathcal{P} S_{\lambda}^{\varepsilon}f(x_t).
\end{align*}
where $x_t = \pi\circ \Phi_{\varepsilon t}(x, g\circ \beta(p))$ (as before,  $\pi : T^*M \to M$ is the canonical projection) and  $U$ is a uniform random variable on $\mathcal{S}^{\mathcal{H}}_{x_t}$. Putting these pieces together,
\begin{align*}
 &\mathbb{E}\bigg[  \int_{\tau_1}^{\infty} e^{-\lambda t} f(B^{\varepsilon}_t(x,p)) dt \bigg]   = \mathbb{E}\left[ e^{-\lambda \tau_1} S^{\varepsilon}_{\lambda} f(x_1^{\varepsilon}, v_1^{\varepsilon}) \right]  = \mathbb{E} \left[ e^{-\lambda \tau_1} \mathcal{P} S_{\lambda}^{\varepsilon}f(x_{\tau_1})\right] \\
 &= \int_0^{\infty} e^{-\lambda t} e^{-t} \mathcal{P} S_{\lambda}^{\varepsilon} f(\Phi_{\varepsilon t}(x,g \circ \beta(p))) dt  = R^{\varepsilon}_{\lambda+1} \mathcal{P} S_{\lambda}^{\varepsilon} f (x,g \circ \beta(p)).
\end{align*}
Note that the third equality used $\mathcal{P} S_{\lambda}^{\varepsilon}f(x_{\tau_1}) = \mathcal{P} S_{\lambda}^{\varepsilon} f(\Phi_{\varepsilon t}(x,g \circ \beta(p)))$ by the identification of $C^{\infty}_c(M)$ as a subset of $C^{\infty}_c(T^*M)$.
\end{proof}

\begin{proof}[Proof of Theorem {\ref{thm.SemiGroupTt}}]
Using Lemmas \ref{lem.PrfSemiGrp1} and \ref{lem.PrfSemiGrp2}, we have
\begin{align*}
S^{\varepsilon}_{\lambda} f(x, p) = \mathbb{E}\bigg[ \int_0^{\tau_1} e^{-\lambda t} f(B^{\varepsilon}_t (x,p)) dt \bigg]  + \mathbb{E}\bigg[  \int_{\tau_1}^{\infty} e^{-\lambda t}  f(B^{\varepsilon}_t (x,p)) dt \bigg] \\
= R^{\varepsilon}_{1+\lambda} f(x,g \circ \beta(p)) +  R^{\varepsilon}_{1+\lambda} \mathcal{P}^g S^{\varepsilon}_{\lambda} f(x, g \circ \beta(p)).
\end{align*}
Multiplying on the left by $1 + \lambda - \varepsilon \mathscr{D}_{HJ}$ yields
\begin{align*}
(1 + \lambda - \varepsilon \mathscr{D}_{HJ}) S^{\varepsilon}_{\lambda} f(x, g \circ \beta(p)) =  f(x, g \circ \beta(p)) +  \mathcal{P}^g S^{\varepsilon}_{\lambda} f(x, g \circ \beta(p)).
\end{align*}
That is,
\[
(\lambda - [\varepsilon \mathscr{D}_{HJ} + \mathcal{P}^g - \mathcal{I}]) S^{\varepsilon}_{\lambda}f(x, g \circ \beta(p)) = f(x, g \circ \beta(p)).
\]
In particular, for any $f \in C^{\infty}_c(M)$,
\[
(\lambda - [\varepsilon \mathscr{D}_{HJ} + \mathcal{P}^g - \mathcal{I}]) S^{\varepsilon}_{\lambda}f = f.
\]
From here we can now conclude the result.
\end{proof}

\subsection{Averaging over the unit sphere in an inner product space}
Here we provide details of the proof of Proposition \ref{prop.unifint} which are solely properties of finite-dimensional inner product spaces.

\begin{prop}
\label{prop.equalint}
Let $\mathcal{X}$ be an $n$-dimensional real inner product space with inner product $\langle \cdot, \cdot \rangle$. Let $S$ be the unit sphere in $\mathcal{X}$ with respect to this inner product and set $\mu$ as the rotationally invariant probability measure on $S$. Given any $X \in \mathcal{X}$,
\[ \int_S (X,\xi)^2 \mu(d\xi) = \frac{|X|^2}{n}. \]
\end{prop}
\begin{proof}
It suffices to show that if $X \in S$, then $\int_S (X,\xi)^2 \mu(d\xi) = 1/n.$ To this end, suppose $X, Y \in S$ and $l : S \to S$ is any rotation such that $l(Y)=X$. Since the adjoint of a rotation is again a rotation, we have,
\[ \int_S (X,\xi)^2\mu(d\xi) = \int_S (l(Y),\xi)^2 \mu(d\xi) = \int_S (Y, l^{\ast}(\xi))^2 \mu(d\xi) = \int_S (Y,\xi)^2 \mu(d\xi)\]
where the final identity follows from the rotational invariance of $\mu$. This shows that the value of the integral is constant for any choice of $X \in S$. Set

\[
a :=  \int_S (X,\xi)^2\mu(d\xi).
\]
Take $\{X_i: 1\leqslant i \leqslant n \} \subset S$ to be an orthonormal basis for $V$, then for any $\xi \in S$

\[
1 = \| \xi \|^2 = \sum_{i=1}^n (X_i,\xi)^2.
\]
Therefore,
\[ 1 = \int_S \| \xi \|^2 \mu(d\xi) = \sum_{i=1}^n \int_S (X_i, \xi)^2 \mu(d\xi) = na \]
which then implies $a = 1/n$.
\end{proof}

\begin{cor}
\label{cor.equalint}
Let $\mathcal{X}$, $S$, and $\mu$ be as in the previous proposition. Take $X,Y \in \mathcal{X}$. Then
\[
\int_S (X,\xi)(Y,\xi) \mu(d\xi) = \frac{(X,Y)}{n}.
\]
\end{cor}
\begin{proof}
By the previous proposition,
\[
\int_S (X+Y,\xi)^2 \,\mu(d\xi) = \frac{|X+Y|^2}{n} = \frac{|X|^2}{n} + \frac{|Y|^2}{n} + 2\frac{(X,Y)}{n}.
\]
On the other hand, $(X+Y,\xi)^2 = (X,\xi)^2 + (Y,\xi)^2 + 2(X,\xi)(Y,\xi)$. Hence another application of  the previous proposition yields,
\begin{align*}
\int_S (X+Y,\xi)^2 \,\mu(d\xi) = \int_S \big\{ (X,\xi)^2 + (Y,\xi)^2 + 2(X,\xi)(Y,\xi) \big\} \mu(d\xi) \\= \frac{|X|^2}{n} + \frac{|Y|^2}{n} + 2\int_S(X,\xi)(Y,\xi)\,\mu(d\xi).
\end{align*}
Comparing terms, the result now follows.
\end{proof}

\begin{acknowledgement}
The authors are grateful for many helpful and motivating conversations with Alexander Teplyaev, Michael Hinz, and Dan Kelleher. In large part, this paper is the result of our attempt to address several questions raised during those discussions.
\end{acknowledgement}

\end{document}